\documentclass[12pt,reqno]{amsart}
\usepackage{amsthm}
\usepackage{multirow}
\usepackage{amsmath}
\usepackage[margin=1in]{geometry}
\usepackage{amssymb}
\usepackage{verbatim}
\usepackage{longtable,enumitem,tikz-cd,color}
\usepackage{stmaryrd}
\usepackage{graphicx}
\usepackage{latexsym}   
\usepackage{amsmath}    
\usepackage{amsbsy}
\usepackage{array}
\usepackage[all]{xy}
\usepackage{hyperref}
\usepackage{caption}
\usepackage{booktabs}
\usepackage{makecell}
\usepackage{adjustbox}

\theoremstyle{plain}
\newtheorem{algorithm}{Algorithm}[section]

\newtheorem{corollary}[algorithm]{Corollary}

\newtheorem{definition}[algorithm]{Definition}
\newtheorem{observe}[algorithm]{Observation}

\newtheorem{lemma}[algorithm]{Lemma}

\newtheorem{main}{Theorem}

\newtheorem*{theorem*}{Theorem}
\newtheorem{theorem} [algorithm] {Theorem}

\newtheorem{proposition}[algorithm]{Proposition}
\newtheorem{remark}[algorithm]{Remark}

\newtheorem{Codim2}[algorithm]{Codimension Two Lemma}
\numberwithin{equation}{algorithm}

\newtheorem*{1/2MSR}{Half-Maximal Symmetry Rank Theorem}
\newtheorem*{corb}{Corollary C}

\newtheorem*{Borel}{Borel Formula}

\newtheorem*{CL}{Connectedness Lemma}
\newtheorem*{PL}{Periodicity Lemma}
\newtheorem*{JB}{Johnson Bound}
\newtheorem*{SE}{Stirling's Inequalities}

\newcommand{\Z}{\mathbb{Z}}
\newcommand{\C}{\mathbb{C}}

\newcommand{\RP}{\mathbb{R}\mathrm{P}}
\newcommand{\CP}{\mathbb{C}\mathrm{P}}
\newcommand{\HP}{\mathbb{H}\mathrm{P}}

\DeclareMathOperator{\codim}{codim}

\newcommand{\ceil}[1]{\left\lceil #1 \right\rceil}
\newcommand{\floor}[1]{\left\lfloor #1 \right\rfloor}

\usepackage[T1]{fontenc}
\usepackage{makecell}

\title{$\Z_p$-torus actions on positively curved manifolds} 
\author{Muhammad Abdullah} 
\address{Department of Mathematics, 
Wichita State University, Wichita, KS 67260 USA}
\email{mxabdullah2@shockers.wichita.edu}
\author{Catherine Searle}\address{Department of Mathematics, Wichita State University, Wichita, KS 67260 USA}\email{catherine.searle@wichita.edu}
\date{\today}
\subjclass{53C21; 20K01}

\begin{document}
\begin{abstract}  

In this article, we study closed, positively curved $n$-manifolds that admit an effective, isometric $\mathbb{Z}_p^r$-action with a fixed point, where $p$ is an odd prime.
For all sufficiently large $n$, we obtain a symmetry-rank bound in Theorem~\ref{main2} that improves the 
$3n/8$ bound of Fang and Rong \cite{FR} and Ghazawneh \cite{Gh}. We improve on this bound for small odd primes $3\leq p\leq 19$ in Theorem~\ref{main}. One of our main tools comes from the theory of error-correcting codes and is of independent interest: we derive  a finite-length Plotkin bound and a finite-length Elias-Bassalygo bound for $q$-ary codes and show that the finite-length Plotkin bound is asymptotically sharper for all primes $p \ge 23$.

\end{abstract}
\maketitle
\vspace{-1cm}
\section{Introduction}

A long-standing problem in Riemannian geometry is the classification of closed manifolds of positive sectional curvature. To date, apart from special examples in dimensions less than or equal to 24, all known examples are spherical in nature and highly symmetric. 
The Symmetry Program suggests that one way to approach this classification is to study manifolds with ``large" symmetries, with the goal of gradually decreasing the amount of symmetry. Abelian groups are a natural choice to study in this context and the case of torus actions has been studied extensively: see, for example, the work of Grove and Searle \cite{GS}, Fang and Rong \cite{FR2}, Wilking \cite{Wilk}, and of Kennard, Wiemeler, and Wilking \cite{KWW1, KWW2}.

Consistent with the goal of the Symmetry Program, it is natural to study $\mathbb{Z}_p^r$-actions. Previous work on positively and non-negatively curved Riemannian manifolds with discrete symmetries of dimension 4 can be found in Yang \cite{Y}, Hicks \cite{Hi}, Fang \cite{F}, and Kim and Lee \cite{KL}, and of higher dimensions in Fang and Rong \cite{FR}, Su and Wang \cite{SW}, Wang \cite{W}, and more recently in Kennard, Khalili Samani, and Searle \cite{KKSS}, and Ghazawneh \cite{Gh1, Gh}. 

In all of these works, with the exception of the last three, the prime $p$ is assumed to be sufficiently large so that $\Z_p^r$ admits a fixed point. 
In contrast, in \cite{KKSS}, \cite{Gh1}, and \cite{Gh}, they consider $\mathbb{Z}_p^r$-actions with the additional assumption that there exists a fixed point. We make the same assumption in this paper, but only consider odd primes. 
The following is our main theorem, giving an improvement on the lower bound on $r$ of approximately $3/8$ of the dimension of the manifold obtained in \cite{FR} and \cite{Gh} for all odd primes, see also Theorem \ref{Gh+Zp}.

\begin{main}
\label{main2}
Assume $\Z_p^r$ acts effectively and isometrically on a closed, $n$-dimensional, positively
curved manifold $M$ with a fixed-point $x\in M$, where $p \ge 3$ and $n \geq 13$. 
Suppose
\begin{equation}\label{eq:plotkin-rank}
r \;>\; F_1(n,p)=\frac{3p-4}{8(p-1)}\,n \;+\; \log_p(n) \;+\; \log_p\!\left(\frac{p^{3}}{8}\right).
\end{equation}
Then $M$ is homotopy equivalent to $S^n$, $\RP^n$, $\CP^{n/2}$, or a lens space.
\end{main}

We further improve the bound for small primes in the next theorem.

\begin{main} \label{main}
Assume $\Z_p^r$ acts effectively and isometrically with a fixed point on a closed, positively
curved, $n$-dimensional manifold $M$, where $p$ is a prime, $3\leq p\leq 19$, and $n \geq 16$.
Suppose
\begin{equation}\label{bound}
r \;>\; F_2(n, p)=f_1(p)\,n \;+\; 2.5 \log_p n \;+\; f_2(p)
   \;+\; \frac{f_3(p)}{\,n-1\,}
   \;+\; \frac{f_4(p)}{\,f_5(p)(n-1) - 2},
\end{equation}
where the
$f_i(p)$, given in Display \ref{constants_f}, are functions depending on $\log_p$ terms and the entropy and Johnson functions, given in Definition \ref{entropy}. 
Then $M$ is homotopy equivalent to $S^n$, $\RP^n$, $\CP^{n/2}$, or a lens space.
\end{main}

We show that the bound in Theorem \ref{main2} is stronger than the one obtained in Theorem \ref{main}  
whenever $p \ge 23$, and for each prime $p \ge 3$ and sufficiently large values of $n$, it eventually becomes stronger than the $3n/8$ estimate  appearing in Theorem~\ref{Gh+Zp}.  We also see, in Proposition \ref{LambertW}, that the approximate crossover occurs at the explicit Lambert-$W$ quantity
\[
N_{0}(p)
=
-\frac{8(p-1)}{\ln p}\,
W_{-1}\!\left(
-\,\frac{\ln p}{(p-1)p^{3}}
\right),
\]
after which the Plotkin inequality is strictly smaller than the bound in
Theorem~\ref{Gh+Zp}.  Thus, for $n \ge \lceil N_{0}(p)\rceil$, the
Plotkin estimate yields an improvement over the 
estimate $3n/8$ appearing in Theorem~\ref{Gh+Zp}.

  By comparison, the bound obtained in Theorem \ref{main} satisfies $$\frac{n}{4}<F_2(n, p)\leq \frac{\sqrt{3}}{4}<\frac{7n}{16},$$ 
for sufficiently large $n$ and any $p$, so the bound is always better than the maximal symmetry rank bound of $n/2$ for $\Z_p^r$ actions in \cite{FR} and \cite{Gh}. 
Moreover, for $3 \le p \le 19$ and for sufficiently large dimensions depending on $p$, we have that $F_2(n, p)\leq c(p)n$ for $n\geq n(p)$. Likewise, for $ 23\leq p\leq 43$ and for $n\geq n(p)$, we have that $F_1(n,p) \leq c(p)n$.  The explicit values of  $c(p)$ and $n(p)$ are given in the following table.

\begin{table}[h] \label{tbl}
\centering
\renewcommand{\arraystretch}{2.0}
\begin{adjustbox}{max width=\textwidth}
\begin{tabular}{|c||c|c|c|c|c|c|c|c|c|c|c|c|c|c|}
\hline
$p$
& 3 & 5 & 7 & 11 & 13 & 17 & 19 & 23 & 29 & 31
& 37 & 41 & 43 & 47 \\
\hline\hline
$c(p)$
& $\frac{9}{32}$
& $\frac{8}{25}$
& $\frac{17}{50}$
& $\frac{71}{200}$
& $\frac{179}{500}$
& $\frac{73}{200}$
& \multicolumn{2}{|c|}{$\frac{37}{100}$}
& \multicolumn{2}{|c|}{$\frac{13}{35}$}
& \multicolumn{4}{|c|}{$\frac{149}{400}$} \\
\hline
$n(p)$
& 1908
& 5177
& 2496
& 4302
& 27{,}390
& 16{,}944
& 2661
& 7608
& 5536
& 8447
& 6400
& 9100
& 12100
& 23332 \\
\hline
\end{tabular}
\end{adjustbox}
\caption{Values for $c(p)$ and $n(p)$ for small primes}
\label{tab:placeholder}
\end{table}

Thus, we obtain the following corollary.

\begin{corb} \label{corb}
Assume $\mathbb{Z}_p^r$ acts effectively and isometrically on a closed, positively
curved, $n$-dimensional manifold $M$ with a non-empty fixed-point set, where $p \leq 47$ is an odd prime.
Suppose further that $n \geq n(p)$ and $$r \geq c(p)\,n$$  
where $c(p)$ and $n(p)$ are  as in Table \ref{tab:placeholder},
then $M$ is homotopy equivalent to $S^n$, $\RP^n$, $\CP^{n/2}$, or a lens space.
\end{corb}

 A few additional remarks are in order concerning $\Z_p^r$-actions. We first address the hypothesis that the $\Z_p^r$-action admits a fixed point. While Lemma 2.1 in \cite{FR} shows if one assumes that the prime $p$ is larger than the total Betti number of a positively curved, even-dimensional manifold,  then a fixed point exists. This result can be considered a $\Z_p$-torus analogue  large $p$ of Berger's theorem for torus actions \cite{Be}. However, the best estimates on the total Betti number are $10^{10n^4}$ (see, for example, Wilking \cite{Wilkingsurvey})  
 where $n$ is the dimension of the manifold, and even the conjectured upper bound of $2^n$, which holds for rationally elliptic spaces, excludes the primes $p\leq 29$ for manifolds of dimensions $\geq 16$.

Unlike in even dimensions, there appears to be no known $\Z_p$-torus  analogue in odd dimensions for the theorem of Sugahara \cite{Su}.
We note that while there are no free $\Z_p$-actions on even-dimensional manifolds of positive curvature for $p\geq 3$ by Synge's theorem, in all odd dimensions such actions exist. Moreover,  there are free isometric $(\Z_3\times \Z_3)$-actions on Eschenburg spaces by work of Grove and Shankar \cite{GSh} and on Baza\u{\i}kin spaces by work of Baza\u{\i}kin \cite{Ba}. As noted in \cite{KKSS},  without the assumption of a non-empty fixed point set,  we lose both the existence of a fixed point set and of an induced $\Z_p$-torus action on any fixed point set component of a $\Z_p$-subtorus containing the fixed point of the $\Z_p^r$.

Finally, we note a few differences between the results to date for $\Z_2$-tori and $\Z_p$-tori. First,  an irreducible representation of $\Z_2$ has real dimension $1$, while that of $\Z_p$ for $p\geq 3$ is $2$, so the $\Z_p$-torus case is more like that of the torus, $T^n$. This accounts for the difference by approximately a factor of two in the maximal and  non-maximal symmetry rank results for these finite tori. Additionally, the lower bound on the rank of $(n+1)/2$ for $n\geq 24$ in Theorem B in \cite{KKSS},  which was improved to $2n/5$ for $n\geq 31, 482$, using Corollary B in \cite{Gh}, corresponds to the half-maximal symmetry rank case. In the case of a torus action, the half-maximal symmetry rank result has a lower bound of $n/4+1$ for $n\geq 10$, which Theorem \ref{main} approximates in the case where $p=3$ and $n\geq 92$. Moreover, as mentioned above, 
our bound in Theorem \ref{main} is no better than the approximately $3n/8$ bound for primes greater than $29$, but for sufficiently large $n$, as noted above, it is better than the  $n/2$ maximal symmetry rank bound. However, as previously remarked, for sufficiently large $n$, our second bound, that is, the bound in Theorem \ref{main2}, \emph{is} better than the approximately $3n/8$ bound for all odd primes. It bears mentioning that Theorem B  in \cite{KKSS} and the half-maximal symmetry rank result in \cite{W} both relied on the binary error correcting code lower bound in \cite{W} that represents a significant improvement for finite-length codes over the finite-length Elias-Bassalygo binary bound. For $q$-ary codes, there is no direct analogue of the binary lower bound in \cite{W}, so one must look for other ways to improve on the finite-length Plotkin and  Elias-Bassalygo $q$-ary bounds we use here.

\subsection{Organization}

The paper is organized as follows. In Section \ref{2}, we derive the  finite-length Elias-Bassalygo and Plotkin $q$-ary bounds and establish various properties of both. 
In Section \ref{3}, we gather 
the necessary geometric and topological background information. In Sections \ref{4} and \ref{5}, we prove Theorems \ref{main} and \ref{main2}, respectively. 

\subsection{Acknowledgments} The authors are grateful to Lee Kennard for helpful conversations. Both authors were partially supported by NSF grants DMS-2204324 and DMS-2506633. 

\section{Error Correcting Codes}\label{2}

In this section, we provide estimates derived from the theory of error correcting codes that have applications to the theory of positively curved manifolds. 

Recall that a $q$-ary linear code is a subspace $\mathcal{C}$ of $\mathbb{F}_q^n$ for some $n \geq 1$, where $\mathbb{F}_q$ is a finite field with $q$ elements. In this case, in error-correcting code theory, one says that the \textit{alphabet} $\mathcal{A}$, that is, a finite set of symbols with at least two elements, is $\mathbb{F}_q$.  Note that  the results of this section hold for all finite fields, but for our purposes, we consider only the case $q=p$, where $p$ is any prime,

We first recall the definitions of the Hamming weight and the Hamming distance.
\begin{definition}
    Let $x \in \mathbb{F}_q^n$. The Hamming weight of $x$, denoted $|x|$ is the number of non-trivial entries in $x$. The Hamming distance between $x, y \in \mathbb{F}_q^n$ is $|x-y|$.
\end{definition} 
\noindent We further recall that the Hamming distance defines a metric on $\mathbb{F}_q^n$, with the induced topology being the discrete topology.

In the theory of error-correcting codes, one typically considers the quantity
$A_q(n,d)$, the maximum number of codewords in a $q$-ary code of length $n$
with a minimum Hamming distance of at least $d$. It is customary to phrase upper bounds in terms of the
\emph{rate upper bound}
\[
R = \frac{1}{n}\log_q A_q(n,d).
\]
The classical bounds on $A_q(n,d)$---for example, the Hamming, Johnson, Plotkin, Griesmer, and Elias--Bassalygo bounds---can all be stated in the form described above. Among these, for small primes, the asymptotic Elias--Bassalygo bound is known to be stronger than most of the classical inequalities, though it remains weaker than the MRRW--2 bound in the binary case $p=2$ \cite{MRRW}, whose derivation requires substantially more sophisticated techniques than we employ here. It is also worth noting that, for large primes, the asymptotic Plotkin bound becomes tighter than the Elias--Bassalygo bound once $p \ge 23$. Because the coding-theory literature is primarily concerned with \emph{asymptotic} bounds, while our geometric applications require strong \emph{finite-length} bounds, we therefore focus on deriving finite-length versions of both the Elias--Bassalygo bound and the Plotkin bound.

In particular, 
we are interested
in upper bounds on the minimum Hamming weight of the nonzero elements in the
image of an injective, linear map
\[
\Z_p^r \hookrightarrow \Z_p^m.
\]

Finally, we note the geometric connection motivating this formulation. If $N$
is a fixed-point component of an element $\tau\in \Z_p^r$ acting on a
positively curved manifold $M$, then $\tau$ acts linearly on $T_xM$ for any
$x\in N$. The injective linear maps above can thus be viewed as arising from
such isotropy representations, and the coding-theoretic bounds translate into
codimension constraints on the fixed-point sets. In particular, for $p=2$, since the irreducible representations of an involution are $1$-dimensional, we set $m=n$, where $n$ is the dimension of our manifold. For $p$ an odd prime, the irreducible representations of $\tau\in\Z_p^r$ are $2$-dimensional, and so we set $m=\lfloor n/2\rfloor$.

\subsection{The Finite-Length Elias-Bassalygo Bound.}

To derive the finite-length Elias-Bassalygo bound, we first collect one definition and three lemmas that are needed for the proof. The definition introduces two standard functions from coding theory: the {\em $q$-ary entropy function} and the {\em Johnson radius}.

\begin{definition}[{\bf $\bf{q}$-ary entropy function and the Johnson radius}] \label{entropy} Let $\delta \in [0,\frac{q-1}{q}]$.
The \emph{$q$-ary entropy function} $H_q(x)$ and the \emph{Johnson radius} $J_q(\delta)$ are respectively defined by
\begin{align*}
H_q(x) &= x \log_q(q-1) - x \log_q x - (1 - x) \log_q(1 - x), \,\,\text{and}\\
J_q(\delta) &= \left(1 - \frac{1}{q} \right) \left(1 - \sqrt{1 - \frac{q}{q-1} \delta} \right).
\end{align*}

\end{definition}

We now fix an alphabet size $q\ge2$, block length $n$, and let $C\subseteq\mathbb{F}_q^n$ be a code of size $|C|$ and minimum distance $\delta n$.  We derive an explicit upper bound on the {\em information rate of $C$}
$$
R_C=\frac{\log_q |C|}{n}
$$
in terms of the $q$-ary entropy $H_q$ and Johnson function $J_q$, that, in turn, gives us an upper bound on $R$ since $C$ is arbitrary. 

Two remarks are in order. First, the function $J_q$, often called the
\emph{Johnson radius} in the coding theory literature, is an increasing
function that maps the interval $[0,\frac{q-1}{q}]$ onto itself.
Second, the function $1-H_q(\epsilon)$ decreases in the same interval,
a fact that is used to derive the final inequality.

We now collect the results needed to derive the Elias-Bassalygo bound. The first is Lemma 2.2 from \cite{RudraLect19}.

\begin{lemma}\label{pigeonhole}\cite{RudraLect19}
For any integer $0\le e\le n$, there exists a 
$y\in\mathbb{F}_q^n$ such that
$$
|C\cap B_q(y,e)|
\ge |C|\;\frac{\mathrm{Vol}_q(0,e)}{q^n},
$$
where
\begin{equation}\label{vol}
\mathrm{Vol}_q(0,e)=\sum_{i=0}^e\binom{n}{i}(q-1)^i.
\end{equation}
\end{lemma}

\begin{proof}
The proof, which is combinatorial,  is sketched in \cite{RudraLect19} with some parts of it given in detail in \cite{RudraLect14}, so we include the proof for the convenience of the reader.

Pick  
$y' \in \mathbb{F}_q^n$ uniformly at random and define
$$
X(y') := \bigl|C \cap B_q(y',e)\bigr|
   = \sum_{c \in C} \mathbf{1}_{\{d(c,y')\le e\}} .
$$
By linearity of expectation,
$$
\mathbb{E}[X]
= \sum_{c \in C} \Pr\,\!\bigl(d(c,y')\le e\bigr).
$$
For any fixed $c \in C$, the probability that $y'$ lies within distance $e$ of $c$, that is, $\mathrm{Pr}\, (d(c,y')\le e)$,
is exactly
$$
\frac{\mathrm{Vol}_q(0,e)}{q^n},
$$
since the Hamming ball of radius $e$ around $c$ has $\mathrm{Vol}_q(0,e)$ points
out of the total $q^n$ points. Hence
$$
\mathbb{E}[X] = \frac{|C| \, \mathrm{Vol}_q(0,e)}{q^n}.
$$
Therefore, by the pigeonhole principle, there must exist some $y \in \mathbb{F}_q^n$ with
$$
|C \cap B_q(y,e)| \;\ge\; \mathbb{E}[X]
= \frac{|C| \, \mathrm{Vol}_q(0,e)}{q^n}.
$$
\end{proof}

The {\em Johnson bound} is another classical bound from the theory of error-correcting codes. We state the version given in Theorem 3.2 in \cite{Guruswami}, as it best suits our purposes. 
\begin{JB}
\cite{Guruswami} \label{johnson}
If the code $C\subseteq\mathbb{F}_q^n$ has minimum distance $d = \delta n$ and one sets
\[
\frac{e}{n}< J_q(\delta).
\]
then for every 
$y\in \mathbb{F}_q^n$,
\[
|C\cap B_q(y,e)| \le q\,n\,\delta n = q\,n\,d.
\]
\end{JB}
\noindent We also need the following classical inequality from analysis, which is proved in \cite{Robbins}.

\begin{SE}\cite{Robbins}\label{Stirling}
For any integer $k\ge1$, one has
\[
\sqrt{2\pi k}\,\bigl(k/e\bigr)^k\,e^{\frac1{12k+1}}
< k!
< \sqrt{2\pi k}\,\bigl(k/e\bigr)^k\,e^{\frac1{12k}}.
\]
Equivalently,
\[
\ln k! = k\ln k - k + \frac12\ln(2\pi k) + \theta_k,
\quad \textrm{where}\,
-\frac1{12k}<\theta_k<\frac1{12k+1}.
\]
\end{SE}

We are now ready to derive the finite-length Elias-Bassalygo bound in the following theorem.

\begin{theorem}[{\bf Finite-Length Elias--Bassalygo Bound}]\label{EB(p)}
Let $C \subseteq \mathbb{F}_q^n$ be a code with size $|C|$ and minimum relative distance $\delta=\frac{d}{n}$, where $d$ is the minimum distance. Set
\[
e :=   \ceil{nJ_q(\delta)}  - 1.
\]
Then we have 
\[
R_C \le 1 - H_q\left(\frac{e}{n}\right)
+ \frac{1}{2n}\log_{q}\left(2\pi \frac{e}{n}(1-\frac{e}{n})n \right)
+\frac{1}{n}\log_{q}\left(qn^2 \delta \right)
\]
\[
+\frac{1}{n \ln q}\left(\frac{1}{12n}+\frac{1}{12e+1}+\frac{1}{12(n-e)+1}\right).
\]
In particular, 
taking the supremum over all such codes $C$, it follows that the rate 
$R= 
\sup(\{R_C : C \subseteq \mathbb{F}_q^n \text{ is a code with minimum relative distance } \delta\})$ satisfies the same inequality.
\end{theorem}

\begin{proof}

We closely follow \cite{RudraLect19}, but give more refined estimates using  Lemma \ref{pigeonhole} and the \hyperref[johnson]{Johnson Bound}. By Lemma \ref{pigeonhole}, there exists a Hamming ball $B$ of radius $e$ such that
$$
   |B| \;\;\ge\;\; \frac{|C| \cdot \mathrm{Vol}_q(0,e)}{q^n}.
$$
The \hyperref[johnson]{Johnson Bound} then gives us that $|B| \leq qnd$. Combining these last two inequalities, we obtain
$$
   |C| \leq qnd \cdot \frac{q^n}{\mathrm{Vol}_q(0,e)}.
$$

Applying $\log_q$ to both sides of this inequality, we see that 
\[
\log_q |C| 
\le n + \log_q(qnd) - \log_q \mathrm{Vol}_q(0,e),\]
and hence, after dividing both sides by $n$, we obtain
\begin{equation}\label{RC} R_C \le 1 + \frac{\log_q(qnd)}{n} - \frac{1}{n} \log_q \mathrm{Vol}_q(0,e).
\end{equation}
It follows from Equation \eqref{vol}  that $\mathrm{Vol}_q(0,e) \ge \binom{n}{e} (q-1)^{e}$. Using \hyperref[Stirling]{Stirling's Inequalities}, we obtain
\[
\log_{q} \binom{n}{e} = n\log_{q}n 
- e \log_{q}e- (n-e) \log_{q}(n-e)
+\frac{1}{2} \log_{q}\left(\frac{n}{2\pi e(n-e)}\right)
- \Delta,
\]
where $|\Delta| \le \frac{1}{\ln q} \left(\frac{1}{12n} + \frac{1}{12e} + \frac{1}{12(n-e)}\right)$. Adding $e \log_{q} (q-1)$ to the above expression and then multiplying the resultant expression by \(-\frac{1}{n}\), it can be checked that
\[
\frac{-\log_q\mathrm{Vol}_q(0,e) }{n}\le -H_{q}(\frac{e}{n})-\frac{1}{2n}\log_{q} \left(\frac{n}{2 \pi e(n-e)} \right)+ \frac{\Delta}{n}.
\]
Substituting this into Inequality \eqref{RC}, and using $d = \delta n$, we obtain the desired result.
\end{proof}

Since floor functions appear frequently in the statement of Theorem~\ref{EB(p)}, they can make direct applications of the bound  somewhat cumbersome.  
To address this issue, we present a continuous relaxation of the finite-length Elias-Bassalygo bound, stated in the following corollary.

\begin{corollary}[Continuous Form of Theorem \ref{EB(p)}]
Fix \( \delta \in \left(0, 1 - \frac{1}{q} \right) \).  
Suppose \( C \subseteq \mathbb{F}_q^n \) is a code of size \( |C| \) and minimum relative distance $\delta=\frac{d}{n}$, where $d$ is the minimum distance.  
Then, for $n >\frac{1}{J_p(\delta)}$, we have $R_C = \frac{\log_q|C|}{n}$ satisfies:
\begin{align}
\begin{split}
R_C &\leq 1 - H_q\left(J_q(\delta)\right)
+ \frac{1}{n} \log_q\left( \frac{(q - 1)(1 - J_q(\delta))}{J_q(\delta)} \right)
+ \frac{1}{2n} \log_q\left( 2\pi n J_q(\delta) \right)
\\
& \,\,\,\,\,\,+ \frac{1}{n} \log_q(qn^2 \delta)
+\frac{1}{12n^2 \ln q}
+\frac{2}{13n \ln q} + \frac{1}{2n^2 \ln q} \left( \frac{1}{J_q(\delta) - \frac{1}{n}} + \frac{1}{1 - J_q(\delta)} \right).
\end{split}
\end{align}
In particular, taking the supremum over all such $C$, it follows that the rate $R$ satisfies the same inequality.
\end{corollary}

\begin{proof}
Using Theorem~\ref{EB(p)} and the fact that $1-H_q(\epsilon)$ is decreasing, we obtain
\begin{align*}
R &\le 1 - H_q\!\left(\frac{e}{n}\right)
  + \frac{1}{2n}\log_{q}\!\left(2\pi \frac{e}{n}\!\left(1-\frac{e}{n}\right)n \right)
  + \frac{1}{n}\log_{q}\!\left(qn^2 \delta \right)
  + \frac{1}{n \ln q}\!\left(\frac{1}{12n}+\frac{1}{12e+1}+\frac{1}{12(n-e)+1}\right)\\
&\le 1 - H_q\!\left(J_q(\delta) - \frac{1}{n}\right)
  + \frac{1}{2n} \log_q\!\bigl( 2\pi n J_q(\delta) \bigr)
  + \frac{1}{n} \log_q(qn^2 \delta)
  + \frac{1}{12n^2 \ln q} + \frac{2}{13n \ln q},
\end{align*}
where $e := \lceil n J_q(\delta) \rceil - 1$.

Since 
$$
\frac{e}{n} \in \Bigl(J_q(\delta) - \frac{1}{n}, \, J_q(\delta)\Bigr],
$$ 
a Taylor expansion of $H_q$ at $J_q(\delta)$ results in
$$
H_q\!\left( J_q(\delta) - \frac{1}{n} \right) 
= H_q\!\bigl(J_q(\delta)\bigr) 
  - \frac{1}{n} H_q'\!\bigl(J_q(\delta)\bigr) 
  + \frac{1}{2n^2} H_q''(\xi),
$$
for some $\xi \in \bigl(J_q(\delta) - \frac{1}{n}, \, J_q(\delta)\bigr)$.

Thus,
$$
1 - H_q\!\left( J_q(\delta) + \frac{1}{n} \right)
\;\le\; 1 - H_q\!\bigl(J_q(\delta)\bigr)
  + \frac{1}{n} H_q'\!\bigl(J_q(\delta)\bigr)
  + \frac{1}{2n^2}\,\bigl| H_q''(\xi) \bigr|.
$$

We compute:
$$
H_q'\!\bigl(J_q(\delta)\bigr) 
  = \log_q\!\left( \frac{(q - 1)(1 - J_q(\delta))}{J_q(\delta)} \right), 
\qquad
H_q''(y) 
  = -\frac{1}{\ln q}\!\left( \frac{1}{y} + \frac{1}{1-y} \right),
$$
so
$$
\bigl|H_q''(\xi)\bigr| 
  \;\le\; \frac{1}{\ln q}\!\left( \frac{1}{J_q(\delta) - \frac{1}{n}} + \frac{1}{1 - J_q(\delta)} \right).
$$
Substituting this into the inequality for $R$ above, we obtain the result.

\end{proof}

Letting $\rho: \Z_p^r\rightarrow \Z_p^n$ 
be an injection that is not necessarily linear, if we set  $C=\rho(\Z_p^r)$, then  $R_C=\frac{r}{n}$. In particular, Corollary \ref{EB(p)} gives us the following version of Theorem 1.2 in \cite{KKSS} for $q$-ary codes with $q=p$. 

\begin{corollary} \label{rEB(p)} 
Fix $\delta\in(0, 1-1/p)$, $p$ a prime. If $\rho: \Z_p^r\rightarrow \Z_p^n$ 
is an injection (not necessarily linear) such that $|\rho(\tau_p)|\geq \delta n$ 
for all non-trivial $\tau_p\in \Z_p^r$. For $n >1/J_p(\delta),$ 
\begin{align*}
r &\leq (1 - H_p(J_p(\delta)))n 
+ 2.5\log_p n 
+ \log_p\!\left( \frac{(p - 1)(1 - J_p(\delta))}{J_p(\delta)} \right)+ 0.5 \log_p\!\left( 2\pi J_p(\delta) \right)  \\
&
\,\,\,\,\,\,\,+ \log_p(p \delta) 
+\frac{2}{13 \ln p}+ \frac{1}{n}\!\left(\frac{1}{12\ln p} +\frac{1}{(2 \ln p)(1 - J_p(\delta))}\right) + \frac{1}{(2\ln p)J_p(\delta)n - 2 \ln p}=r(\delta).
\end{align*}
\end{corollary}

\begin{remark}
    Note that for large $n$, the dominant term in the  bound in Corollary \ref{rEB(p)} is $(1-H_p(J_p(\delta)))n$. In the proof of Theorem \ref{main}, we set $\delta=\frac{1}{4}$ and replace n in the linear (dominant) term by $\frac{n}{2}$.   We can then view $f_1(p) = \frac{1}{2}(1-H_p(J_p(1/4)))$, 
    as a function in $p$. In the following proposition, we show that $f_1(p) >\frac{3}{8}$ for $p \geq 31$, and thus Corollary \ref{rEB(p)}  yields  no improvement of the best known lower bound for the rank of the $\Z_p$-torus given in Theorem \ref{Gh+Zp} for $p\geq 31$. 
    \end{remark}

\begin{proposition}\label{p29}
    The function $f_1(p) = \frac{1}{2}(1-H_p(J_p(1/4))$ is strictly increasing for $p \ge 3$, and, in particular, $f_1(p) \geq f_1(31) >3/8$ for $p \geq 31$.
\end{proposition}

\begin{proof}
It suffices to show that $f(p)=2f_1(p)=1-H_p(J_p(1/4))$ is strictly increasing for $p \geq 3.$

Let $x(p):=J_p(\frac14)=(1-\frac1p)\bigl(1-s(p)\bigr)$ where 
$$
s(p)=\sqrt{\frac{3p-4}{4(p-1)}}\in(0,1)\quad (p \ge2).
$$
We first show that $x(p)$ is strictly decreasing and $x(p)\in(0, 1/2)$ for $p\geq 3$. Differentiating, we obtain
$$
x'(p)=\frac{1-s(p)}{p^2}-\Bigl(1-\frac1p\Bigr)\frac{1}{8(p-1)^2\,s(p)}.
$$
All denominators are $>0$, so $x'(p)<0$ is equivalent to 
$8(p-1)s(p)(1-s(p))<p$. Substituting in the expression for $s(p)$ and rearranging gives us that 
we need $4\sqrt{(p-1)(3p-4)}<7p-8$, which, for $p \ge 3$, is equivalent, after squaring both sides and rearranging again, to $0<p^2$. 
Hence $x'(p)<0$ for all $p \ge 3$. Since $x(3)\leq 0.14$, the claim holds.

Now $f(p)=1-H_p(x(p))$, so by the chain rule
$$
f'(p)=-\frac{d}{dp}\big(H_p(x(p))\big)
= -\Bigl(\partial_p H_p\Bigr)(x(p)) \;-\; \Bigl(\partial_x H_p\Bigr)(x(p))\,x'(p).
$$
We claim that the first term is strictly positive  and the second is non-negative, which in turn gives us that $f'(p)>0$. We first consider $\partial_x H_p(x)$. A direct calculation gives
\begin{align*}
\partial_x H_p(x)&=\frac{1}{\ln p}\,\ln\,\!\left(\frac{(p-1)(1-x)}{x}\right).\\
\end{align*}

\noindent For $p \ge 3$ we have $0<x(p)<(p-1)/p$, hence the argument in the righthand side of the above display is strictly greater than $1$ and  
therefore $\partial_x H_p(x(p))>0$. 
Since  $x'(p)<0$, we conclude that  $-\bigl(\partial_x H_p\bigr)(x(p))\,x'(p)>0$, as claimed.
\vspace{.1 cm}

We now consider $\partial_p H_p(x)$.
Another direct computation gives
\begin{align}\label{eq:dHpdp-exact}
\begin{split}
\partial_p H_p(x)&
=\frac{x\,p\ln p-(p-1)\left(x\ln(p-1)-x\ln x-(1-x)\ln(1-x)\right)}{p(p-1)(\ln p)^2}\\
&\leq \frac{x\,p\ln p-(p-1)\left(x\ln(p-1)+2x(1-x)\right)}{p(p-1)(\ln p)^2},
\end{split}
\end{align}
since, for \(x\in(0,1/2)\) we have
$$
-\,x\ln x\ge x(1-x)\,\,\,\, {\textrm{and} }\,\, -\,(1-x)\ln(1-x)\ge x(1-x)$$.

In particular, for $x\le \frac12$, for the numerator of the second expression in Inequality \eqref{eq:dHpdp-exact} we see that
$$
x\,p\ln p-(p-1)(x\ln(p-1)-x\ln x-(1-x)\ln(1-x))
\;\le\;x\,g(p),
$$
where
$$
g(p):=p\ln p-(p-1)\ln(p-1)-(p-1).$$ Then
$$
g'(p)=\ln\!\left(1+\frac{1}{p-1}\right)-1<\frac{1}{p-1}-1<0,
$$
hence $g$ is strictly decreasing on $[3,\infty)$. A direct check shows
$$
g(3)=2\ln\left(\frac32\right)+\ln 3-2<0,
$$
so $g(p)<0$ for all $p\ge 3$. By \eqref{eq:dHpdp-exact} this implies
$$
\partial_p H_p(x)\;<\;0\quad\text{for all }p\ge 3\text{ and }x\le\frac12.
$$
In particular, since $x=x(p)=J_p(1/4)\leq \frac{1}{2}$, we have
$$
-\bigl(\partial_p H_p\bigr)\bigl(x(p)\bigr)\;\ge\;0.
$$

Altogether,
$$
f'(p)=\underbrace{\bigl[-(\partial_x H_p)(x(p))\,x'(p)\bigr]}_{>0}\;+\;\underbrace{\bigl[-(\partial_p H_p)(x(p))\bigr]}_{\ge 0}\;>\;0
\qquad\text{for all }p>2.
$$
Thus $f$ is strictly increasing on $[3,\infty)$. In particular, a direct calculation shows that $f(31)>3/8$ and so the result follows.
\end{proof}

We also need to use the following technical proposition concerning our bound in the proof of Lemma \ref{lemmaA}. Set $r_{EB}(\delta)$ equal to the right-hand side of the displayed inequality in Corollary \ref{rEB(p)}.

\begin{proposition}[{\bf Monotonicity of the Bound at $\delta=\frac14$ and $\delta=\frac13$}]\label{prop:bound-decreasing}
Let $p \ge 3$ and $n \ge 16$. Then the upper bound $r_{EB}(\delta)$ in Corollary \ref{rEB(p)} satisfies
\[
r_{EB}\!\left(\frac13\right) < r_{EB}\!\left(\frac14\right).
\]
\begin{proof} 
We first note that 
$$J' = J_p'(\delta) = \frac{1}{2}\left(1 - \frac{p}{p-1}\,\delta\right)^{-\frac{1}{2}} > 0 \,\,\,\textrm{and}\,\,\, J_p''(\delta) = \frac{p}{4(p-1)}\left(1 - \frac{p}{p-1}\,\delta\right)^{-\frac{3}{2}} > 0 .$$ 
Since $J(0)=0$ and 
$J((p-1)/p)=(p-1)/p$, 
we see that $J(\delta)\leq \delta$ on 
$\left[0, \frac{p-1}{p}\right]$.  
Moreover, $J'(\delta)\geq 1/2$, and so $J(\delta)\geq \delta/2$.

Differentiating $r_{EB}(\delta)$ term-by-term gives
\[
\begin{aligned}
\frac{d{r_{EB}}}{d\delta}
&= -\frac{nJ'}{\ln p}\,\ln(p-1)
-\frac{nJ'}{\ln p}\,\ln(1-J)
+\frac{nJ'}{\ln p}\,\ln J -\frac{J'}{(\ln p)(1-J)} \\
&\quad -\frac{J'}{(\ln p)\,J} + \frac{J'}{(2\ln p)\,J} + \frac{1}{\delta\,\ln p} + \frac{J'}{2n(\ln p)(1-J)^2} - \frac{(2\ln p)nJ'}{\bigl((2\ln p)(Jn-1)\bigr)^2}.
\end{aligned}
\tag{$*$}
\]

\noindent As the fourth, fifth, and last terms are negative, we see that 
\[
\frac{d{r_{EB}}}{d\delta} \le -\frac{nJ'}{\ln p}\,\ln\!\frac{(p-1)(1-J)}{J} + \frac{J'}{2\,(\ln p)\,J} + \frac{1}{\delta\,(\ln p)} + \frac{J'}{2n(\ln p)(1-J)^2}.
\]
Since $p\ge3$, and recalling that $J$ is  increasing on $[0, \frac{p-1}{p}]$, we have $$\ln\left(\frac{(p-1)(1-J)}{J}\right) \ge \ln\left(\frac{2(1-\delta)}{\delta}\right).$$ Moreover, as we saw above,  
$\delta/2 \le J \le \delta$ and hence $1-J \ge 1-\delta$. Also, $J' = 1/(2\sqrt{1-\frac{p}{p-1}\,\delta}) \ge 1/(2\sqrt{1-\delta})$. Multiplying by $\ln p>0$, substituting $J \ge \delta/2$ and $(1-J) \ge 1-\delta$, and using that $J' \leq 1$ for $\delta \in [\frac{1}{4},\frac{1}{3}]$ gives
\begin{align*}
(\ln p)\, {r_{EB}}'(\delta) &\le - n J' \ln\!\frac{2(1-\delta)}{\delta} + \frac{J'}{2J} + \frac{1}{\delta} + \frac{J'}{2n(1-\delta)^2}\\
& \le- \frac{n}{2\sqrt{1-\delta}} \ln\!\frac{2(1-\delta)}{\delta} + \frac{2}{\delta} + \frac{1}{2n(1-\delta)^2}.
\end{align*}
For $\delta\in[\frac14,\frac13]$ we have $\frac{1}{\sqrt{1-\delta}} \le \frac{1}{\sqrt{2/3}}$, $\ln\!\frac{2(1-\delta)}{\delta} \ge \ln 4$, $\frac{1}{\delta} \le 4$, and $(1-\delta) \ge \frac23$. Hence
\[
(\ln p)\, {r_{EB}}'(\delta) \le - \frac{n}{2} \ln 4 + 8 + \frac{9}{8n} < 0 \,\,\, \textrm{provided}\,\,\, n \ge 16,
\]
so ${r_{EB}}'(\delta)<0$ for all $\delta\in[\frac14,\frac13]$. By the Mean Value Theorem there exists $\xi\in(\frac14,\frac13)$ such that
\[
r_{EB}\!\left(\frac13\right) - r_{EB}\!\left(\frac14\right) = \frac1{12}\,{r_{EB}}'(\xi) < 0,
\]
which proves $r_{EB}(\frac13) < r_{EB}(\frac14)$.
\end{proof}
\end{proposition}

\subsection{The Finite-Length Plotkin Bound.}

We now introduce another upper bound that is especially useful for large
primes.  For every prime $p \ge 23$, by comparing the dominant linear terms, we see that the Plotkin
bound gives a better asymptotic estimate than the Elias-Bassalygo bound.  In addition, for every prime $p \ge 3$, the
Plotkin bound in this range also represents an improvement on the  $3n/8$ estimate in Theorem \ref{Gh+Zp} for sufficiently large $n$. 
Our goal in this subsection is to derive the finite-length Plotkin bound for $q$-ary codes and then establish the threshold dimensions for when the corresponding bound on the rank is better than the bound in Theorem \ref{Gh+Zp}. 
We first recall the classical Plotkin bound below in the
form given in \cite{RudraLect16}. 

\begin{theorem}[\cite{RudraLect16}]
\label{classical-plotkin}
Let $C \subset [q]^n$ be a $q$-ary code of block length $n$ and minimum Hamming
distance $d$.
\begin{enumerate}
\item If $d = (1 - \frac{1}{q})n$, then
\[
|C| \;\le\; 2qn.
\]
\item If
\[
d > \Bigl(1 - \frac{1}{q}\Bigr)n,
\]
then
\[
|C|
\;\le\;
\frac{q d}{\,q d - (q-1)n\,}.
\]
\end{enumerate}
\end{theorem}

The second part is the one relevant to our setting.  We now combine
Theorem~\ref{classical-plotkin} with a shortening argument of
\cite{RudraLect16} to derive a fully explicit finite-length
Plotkin inequality.

\begin{theorem}[{\bf Finite-Length Plotkin Bound}]
\label{plotkin(p)}
Let $C \subseteq \mathbb{F}_q^n$ be a code with size $|C|$ and minimum relative distance $\delta=\frac{d}{n}$, where $d$ is the minimum distance.
Then
\[
R_C
=
\frac{1}{n}\log_q |C|
\;\le\;
1
-
\frac{q}{q-1}\,\delta
+
\frac{3 + \log_q(\delta n)}{n}.
\]
In particular, taking the supremum over all such $C$, it follows that the rate satisfies the same inequality. 
\end{theorem}

\begin{proof}
Set
\[
n'
:=
\left\lfloor \frac{q}{q-1} d \right\rfloor - 1.
\]
For each $x \in [q]^{\,n-n'}$ define the shortened code
\[
C_x
=
\{ (c_{n-n'+1},\dots,c_n) : c \in C,\; (c_1,\dots, c_{n-n'}) = x \}
\subset [q]^{n'}.
\]
Each $C_x$ has minimum distance at least $d$ (any two elements coincide on
their first $n-n'$ coordinates).  Moreover,
\[
n' < \frac{q}{q-1}d
\quad\Longrightarrow\quad
d > \Bigl(1 - \frac{1}{q}\Bigr)n'.
\]
Thus, applying Part (2) of Theorem~\ref{classical-plotkin} to $C_x$ gives
\[
|C_x|
\;\le\;
\frac{q d}{q d - (q-1)n'}.
\]
Since the denominator is a positive integer,
\[
|C_x| \le qd.
\]
Summing over all $x \in [q]^{\,n-n'}$ yields
\[
|C|
\;\le\;
q^{\,n-n'} \cdot q d.
\]
Since $n' \ge \frac{q}{q-1}d - 2$, we obtain
\[
n - n'
\;\le\;
n - \frac{q}{q-1}d + 2.
\]
Hence
\[
|C|
\;\le\;
q^{\,n - \frac{q}{q-1}d + 2} \cdot qd.
\]
Writing $qd = q^{1 + \log_q(\delta n)}$ with $d = \delta n$, we obtain
\[
|C|
\;\le\;
q^{\,n - \frac{q}{q-1}\delta n + 3 + \log_q(\delta n)}.
\]

The bound for $R_C$ can be seen by first applying $\log_q$ to the above inequality and then dividing by $n$.
\end{proof}

We now apply this to embeddings of $\mathbb{Z}_p^r$ in $\mathbb{Z}_p^n$, giving
the Plotkin analog of Corollary \ref{rEB(p)}.

\begin{corollary}
\label{rPlotkin(p)}
Fix $\delta\in(0, 1-1/p)$, $p$ a prime. If $\rho: \Z_p^r\rightarrow \Z_p^n$ 
is an injection (not necessarily linear) such that $|\rho(\tau_p)|\geq \delta n$ 
for all non-trivial $\tau_p\in \Z_p^r$.
Then
\[
r
\;\le\;
\Bigl(1 - \frac{p}{p-1}\,\delta\Bigr)n
+
\log_p(n) + \log_p(\delta) +3.
\]
\end{corollary}

Just as in the case for the finite-length Elias-Bassalygo bound, we also need to use the following technical proposition concerning our bound in the proof of Lemma \ref{lemmaB}. Set $r_{Pl}(\delta)$ equal to the right-hand side of the displayed inequality in Corollary \ref{rPlotkin(p)}.

\begin{proposition}[{\bf Monotonicity of the Plotkin bound at $\delta=\tfrac14$ and $\delta=\tfrac13$}]
\label{monotone2}
Let $p \ge 3$ and $n \ge 3$. Then
\[
r_{\mathrm{Pl}}\!\left(\frac{1}{3}\right)
<
r_{\mathrm{Pl}}\!\left(\frac{1}{4}\right).
\]
\end{proposition}

\begin{proof}
By definition,
\[
r_{\mathrm{Pl}}(\delta)
=
\Bigl(1 - \frac{p}{p-1}\,\delta\Bigr)n
+
\log_p(n) + \log_p(\delta) + 3.
\]
Thus
\begin{align*}
r_{\mathrm{Pl}}\!\left(\frac{1}{3}\right)
-
r_{\mathrm{Pl}}\!\left(\frac{1}{4}\right)
&=
-\frac{p}{p-1}\Bigl(\frac{1}{3} - \frac{1}{4}\Bigr)n
+
\log_p\!\left(\frac{4}{3}\right) \\
&=
-\frac{p}{p-1}\,\frac{n}{12}
+
\log_p\!\left(\frac{4}{3}\right).
\end{align*}
For $p \ge 3$ we have $\dfrac{p}{p-1} \ge \dfrac{3}{2}$, hence
\[
-\frac{p}{p-1}\,\frac{n}{12}
\le
-\frac{3}{2}\cdot\frac{n}{12}
=
-\frac{n}{8},
\]
and, using $\ln p \ge \ln 3$,
\[
\log_p\!\left(\frac{4}{3}\right)
=
\frac{\ln(4/3)}{\ln p}
\le
\frac{\ln(4/3)}{\ln 3}
=
\log_3\!\left(\frac{4}{3}\right)
< 0.3.
\]
Thus, for $n \ge 3$,
\[
r_{\mathrm{Pl}}\!\left(\frac{1}{3}\right)
-
r_{\mathrm{Pl}}\!\left(\frac{1}{4}\right)
\le
-\frac{3}{8}
+
\log_3\!\left(\frac{4}{3}\right)
< -\frac{3}{8} + 0.3 < 0.
\]
In all cases $p \ge 2$, $n \ge 3$, we therefore have
$r_{\mathrm{Pl}}(1/3) < r_{\mathrm{Pl}}(1/4)$, as claimed.
\end{proof}

Before establishing the thresholds  for the Plotkin bound,  we recall the definition of the Lambert $W$-function,  a special function that naturally arises in the proof of Proposition \ref{LambertW} when calculating these thresholds.

\begin{definition}[{\bf The Lambert $W$--function}]

The \emph{Lambert $W$--function} is defined implicitly as the (multivalued) inverse
of the map $x \mapsto x e^{x}$, that is, $W(z)$ satisfies
\[
W(z)e^{W(z)} = z, \quad z\in \C.
\]
\end{definition}
\begin{remark}
\label{remark:LambertW}

For real arguments $z \in [-e^{-1},0)$, the above equation has exactly two real solutions, denoted by $W_0(z)$ and $W_{-1}(z)$, corresponding respectively to the principal branch and the lower branch. These satisfy
\[
W_{-1}(z) \le -1 \le W_0(z) < 0.
\]
The Lambert $W$--function naturally arises when solving equations of the form
\[
x e^{x} = c.
\]
For a detailed treatment of the Lambert W-function, the reader may consult \cite{CorlessLambertW}.
\end{remark}

Finally, we prove the following technical proposition that anchors the induction in the proof of Theorem \ref{main2}. As a bonus, one of the quantities in the proposition gives an estimate for when the bound in Theorem \ref{main2} is tighter than the bound in Theorem \ref{Gh+Zp}.

\begin{proposition}[{\bf Lambert--$W$ thresholds for the Plotkin bound}]
\label{LambertW}
Let $p\ge3$ be a prime and, for $n\ge13$, let
\[
\widetilde{r}_{\mathrm{Pl}}(n,p)
:=\frac{3p-4}{8(p-1)}\,n+\log_p n+\log_p\!\Bigl(\frac{p^3}{8}\Bigr).
\]

Then the following hold:
\begin{enumerate}
\item $\widetilde{r}_{\mathrm{Pl}}(n,p)\le \frac{3}{8}n$ for all $n\ge N_0(p)$, where $N_0(p)$ corresponds to the larger of the two solutions of $\widetilde{r}_{\mathrm{Pl}}(n,p)=\frac{3}{8}n$; and
\item $\widetilde{r}_{\mathrm{Pl}}(n,p)\ge \frac{3}{8}n+2$ precisely for
$n\in[\,n_{-}(p),\,n_+(p)\,]$, where $0<n_{-}(p)<13<n_+(p)$ are the two solutions of
$\widetilde{r}_{\mathrm{Pl}}(n,p)=\frac{3}{8}n+2$.
\end{enumerate}
\end{proposition}

\begin{proof}
Define
\[
F_p(n):=\frac{3}{8}n-\widetilde{r}_{\mathrm{Pl}}(n,p),
\qquad
G_p(n):=\frac{3}{8}n+2-\widetilde{r}_{\mathrm{Pl}}(n,p).
\]
A direct computation shows
\[
F_p'(n)=G_p'(n)=\frac{1}{8(p-1)}-\frac{1}{n\ln p}
\quad \textrm{and} \quad
F_p''(n)=G_p''(n)=\frac{1}{n^2\ln p}>0,
\]
so both functions are strictly convex on $(0,\infty)$ and therefore have at most
two positive zeros.

\smallskip
\noindent
{\bf Crossover with ${\bf 3n/8}$:}
The equation $F_p(n)=0$ is equivalent to
\begin{equation}\label{cross1}
\log_p\!\Bigl(\frac{p^3}{8}n\Bigr)=\frac{1}{8(p-1)}\,n.
\end{equation}
Multiplying by $\ln p$ and setting
\[
\alpha=\frac{\ln p}{8(p-1)} \quad and \quad x=-\alpha n,
\]
Equation \ref{cross1} becomes
\[
x e^x=-\frac{\ln p}{p^3(p-1)}.
\]
Since $-\frac{1}{e}< -\frac{\ln p}{p^3(p-1)}< 0$ 
for all $p\ge3$, this equation has two
real solutions given by
\[
x=W_0\!\left(-\frac{\ln p}{(p-1)p^3}\right)
\quad\text{and}\quad
x=W_{-1}\!\left(-\frac{\ln p}{(p-1)p^3}\right).
\]
The corresponding positive solutions in $n=-x/\alpha$ yield two crossover points,
with the larger one equal to
\[
N_0(p)
=
-\,\frac{8(p-1)}{\ln p}\,
W_{-1}\!\left(-\frac{\ln p}{(p-1)p^3}\right).
\]
Since $F_p$ is strictly convex and satisfies
$F_p(n)\to+\infty$ as $n\to0^+$ and $n\to\infty$, it follows that
\[
F_p(n)\ge0 \quad\text{for all } n\ge N_0(p),
\]
that is,\ $\widetilde{r}_{\mathrm{Pl}}(n,p)\le\frac{3}{8}n$ for all $n\ge N_0(p)$.

\smallskip
\noindent
{\bf Crossover with ${\bf 3n/8+2}$:}
Similarly, the equation $G_p(n)=0$ is equivalent to
\[
\log_p\!\Bigl(\frac{p^3}{8}n\Bigr)=\frac{1}{8(p-1)}\,n+2,
\]
which, after the same substitution, becomes
\[
x e^x=-\frac{\ln p}{p(p-1)}.
\]
Since $-\frac{\ln p}{p(p-1)}\in(-1/e,0)$ for all $p\ge3$, this equation again has
two real solutions,
\[
x=W_0\!\left(-\frac{\ln p}{p(p-1)}\right)
\quad\text{and}\quad
x=W_{-1}\!\left(-\frac{\ln p}{p(p-1)}\right),
\]
corresponding to the two positive roots
\[
n_{-}(p)
=
-\,\frac{8(p-1)}{\ln p}\,
W_0\!\left(-\frac{\ln p}{p(p-1)}\right),
\qquad
n_+(p)
=
-\,\frac{8(p-1)}{\ln p}\,
W_{-1}\!\left(-\frac{\ln p}{p(p-1)}\right),
\]
with $0<n_{-}(p)<n_+(p)$. By the strict convexity of $G_p$, its sign is negative
precisely on the interval $[\,n_{-}(p),\,n_+(p)\,]$, yielding the stated
inequalities. 

Using the monotonicity of the Lambert--$W$ expressions in $p$, it can be checked that the smaller root $n_{-}(p)$ is decreasing and satisfies $n_{-}(3)<13$, while the larger root $n_-(p)$ is increasing with $\lfloor n_-(3)\rfloor=39$; hence
\[
n_{-}(p) < 13 < n^{*}(p)
\]
for all odd primes $p\ge3$.

\end{proof}

\section{The Topology and Geometry of Positive Curvature}\label{3}

In this section, we gather some of the topological and geometric results used in the proofs of Theorems \ref{main2} and \ref{main}, as well as bounds on the symmetry rank of a $\Z_p$-torus action. 
We begin by recalling the Connectedness Lemma from \cite{Wilk}.

\begin{CL} \label{CL}\cite{Wilk}
    Let $M^n$ be a closed Riemannian manifold with positive sectional curvature. The following hold:
    \begin{enumerate}
        \item {If $N^{n-k}$ is a closed, totally geodesic submanifold of $M$, then the inclusion map $N^{n-k} \hookrightarrow (n-2k+1)$-connected.}
        \item{If $N_1^{n-k_1}$ and $N_2^{n-k_2}$ are closed, totally geodesic submanifolds of $M$ with $k_1 \leq k_2$ and $k_1+k_1 \leq n$, then the inclusion $N_1^{n-k_1} \cap N_2^{n-k_2} \hookrightarrow N_2^{n-k_2}$ is $(n-k_1-k_2)$-connected. }
    \end{enumerate}
\end{CL}

The next result, a reformulation of the proof of Lemma 6.2 from \cite{Wilk}, is an application of the Connectedness Lemma. 

\begin{theorem}\label{cohm1}\cite{Wilk}
    Let $N^{n-k} \subseteq M^n$ be a totally geodesic embedding of a closed, simply connected, positively curved manifold such that $k \leq \frac{n+3}{4}$. If $N$ has the cohomology ring of a sphere, complex projective space, or quaternionic projective space, then the same holds for $M$. Moreover, $k$ is even in the case of complex projective spaces, and $k$ is divisible by four in the case of quaternionic projective spaces.
\end{theorem}

Before presenting a few more useful results, we recall the definition of periodic cohomology and the Periodicity Lemma, see Lemma 2.2 of \cite{Wilk}. 

\begin{definition}[\bf{$\bf{k}$-periodicity}]\label{periodicity}
    Let $M^n$ be a connected, $n$-dimensional manifold. We say that $x$ induces {\em $k$-periodicity}, or $M^n$ has \text{$k$-periodic cohomology} provided there exists $x \in H^k(M^n;\mathbb{Z})$ such that the homomorphism $H^i(M^n; \Z) \xrightarrow{\text{$\cup_x$}}H^{i+k}(M^n; \Z)$ is surjective for $0 \leq i < n- k$ and injective for $0 < i \leq n- k$.
\end{definition}

\begin{PL} \label{PL} \cite{Wilk}
    If an inclusion $N^{n-k} \hookrightarrow M^n$ of closed, orientable manifolds is $(n-k-\ell)$-connected with $n-k-2\ell>0$, then there exists $e \in H^k(M;\Z)$ such that the homomorphism $H^i(M^n; \Z) \xrightarrow{\text{$\cup_e$}}H^{i+k}(M^n; \Z)$ is a surjection for $i=\ell$, and injection for $i+k=n-\ell$, and an isomorphism everywhere in between.
\end{PL}

 The first one of these, Proposition 3.3 in \cite{KKSS}, is a variation on the situation above in which one assumes the existence of two closed, totally geodesic submanifolds, $N_1$ and $N_2$, one of which has periodic cohomology. By making additional assumptions on the codimensions of $N_1$ and $N_2$, we may conclude that $M^n$ also has periodic cohomology. 

\begin{proposition} \label{cohm2} \cite{KKSS}
    Let $M^n$ be a closed, simply connected, positively curved Riemannian manifold. Suppose that $N_1^{n-k_1}$ and $N_2^{n-k_2}$ are two closed, totally geodesic submanifolds of $M$ with $k_1 \leq k_2$. Let $R$ be $\Z$ or $\Z_p$, for some prime $p$. The following hold:
    \begin{enumerate}
        \item {Suppose that $4k_1 \leq n+3$ and $k_1+2k_2 \leq n+1$. If $N_2$ is an $R$-cohomology sphere, then so is $M$.}
        \item {Suppose that $4k_1 \leq n+3$ and $g+k_1+2k_2\leq n+3$ for some $g \geq 2$ that divides $k_1$. If $N_2$ has $g$-periodic $R$-cohomology, then so does $M$.}
    \end{enumerate}
\end{proposition}

Since Theorem \ref{ucover} and Proposition \ref{cohm2} involve simply connected manifolds,
we now recall Theorem 3.4 from \cite{KKSS}, which shows that the information on the fundamental group of the manifold and the cohomology ring of its universal cover can lead to information about homotopy equivalence. 

\begin{theorem} \cite{KKSS} \label{ucover}
    Let $M^n$ be a closed, smooth manifold, and let $\widetilde{M}^n$ denote the universal cover. The following hold:
    \begin{enumerate}
        \item {If $\pi_1(M)$ is cyclic and $\widetilde{M}$ is a cohomology sphere, then $M$ is homotopy equivalent to $S^n$, $\RP^n$, or $S^n/\Z_l$ for some $l \geq 3$. In the last case, $n$ is odd; and}
        \item {If $M$ is simply connected and has the cohomology of $\CP^{\frac{n}{2}}$, then $M$ is homotopy equivalent to $\CP^{\frac{n}{2}}$.}
    \end{enumerate}
\end{theorem}

Combining Proposition \ref{cohm2} and Theorem \ref{ucover} we obtain the following lemma.
\begin{lemma}\label{combo} Let $M^n$ be a closed,  positively curved Riemannian $n$-manifold, $n\geq 5$. Suppose that $N_1^{n-k_1}$ and $N_2^{n-k_2}$ are two closed, totally geodesic submanifolds of $M$ with $ k_1\leq k_2$, $4k_1\leq n+3$, and $k_1+2k_2\leq n+1$. Then 
\begin{enumerate}
\item If $\widetilde{N_2}$ is a cohomology sphere and $N_2$ has cyclic fundamental group, then  
 $M$ is homotopy equivalent to $S^n$, $\RP^n$, or a lens space; and
 \item If  $\widetilde{N_2}$ is cohomology complex projective space,  $2|k_1$ and $\pi_1(N_2)$ is trivial, then $M$ is homotopy equivalent to $\CP^{n/2}$.
\end{enumerate}
  \end{lemma}
  \begin{proof}
We lift to the universal cover of $M$, 
$\widetilde M$. By the \hyperref[CL]{Connectedness Lemma} the corresponding lifts of $N_1$ and $N_2$ are both at least $2$-connected, and therefore their corresponding lifts to $\widetilde M$ are their universal covers, $\widetilde{N_1}$ and $\widetilde{N_2}$.
By Proposition \ref{cohm2}, it follows that $\widetilde M$ is a cohomology sphere or complex projective space. 
Since $\pi_1(N_2)$ is a finite cyclic group and $N_2$ is at least $2$-connected in $M$, it follows by Theorem \ref{ucover}
 that if $\widetilde M$ is a cohomology sphere then $M$ is homotopy equivalent to $S^n$, 
 $\RP^n$, or a lens space. Likewise, if $\widetilde M$ is a cohomology $\CP^{n/2}$, since $N_2$ is assumed to be  simply connected and is at least $(\frac{n-1}{2})$-connected, then $M$ is simply connected and by Theorem \ref{ucover}, $M$ is homotopy equivalent to 
 $\CP^{n/2}$.

\end{proof}

A similar proof allows us to obtain the following generalization of Theorem \ref{cohm1}
 to non-simply connected manifolds as follows.
\begin{lemma}\label{corwilk}
Let $N^{n-k} \subseteq M^n$ be a totally geodesic embedding of closed,  positively curved manifolds such that $k \leq \frac{n+3}{4}$ and the inclusion $N\hookrightarrow M$ is at least $2$-connected. If $N$ has the cohomology of $S^{n-k}$, $\RP^{n-k}$, $\CP^{(n-k)/2}$,  a lens space, or $\HP^{(n-k)/4}$, then the same holds for $M$, where  $k$ is even in the case of complex projective spaces, and $k$ is divisible by four in the case of quaternionic projective spaces. In particular, in the first four cases, $M$ is homotopy equivalent to $S^{n}$, $\RP^{n}$, $\CP^{n/2}$, or  a lens space, respectively.
\end{lemma}
The following lemma, from \cite{KKSS}, gives criteria to obtain a homotopy equivalence classification of fixed point sets of codimension two in closed positively curved manifolds. In particular, the result assuming Condition 1 was proven by Frank, Rong, and Wang \cite{FRW} and that assuming Condition 2 was proven in \cite{KKSS}. The result assuming Condition 3 was originally stated in \cite{KKSS} for actions of $\Z_2$-tori, but the proof, which we leave to the reader, easily generalizes to $\Z_p$-tori, with $p$ a prime.
\begin{Codim2}\label{codim2}
Let $M^{m+2}$ be a closed, positively curved Riemannian manifold. Suppose $N^m$ is a closed, totally geodesic submanifold of $M$.
Assume one of the following holds:
	\begin{enumerate}[font=\normalfont]
	\item \cite{FRW} $M$ is odd-dimensional, of dimension $\geq 5$; 
	\item \cite{KKSS} The inclusion $N^m \subseteq M^{m+2}$ is $m$-connected; or
	\item  $M^{m+2}$ admits a $\Z_p^2$ action such that $N^m$ is a fixed-point set component of a $\Z_p$ subgroup and such that there is a second $\Z_p$ subgroup whose fixed-point set component $N'$ has codimension  at most $\frac{m+1}{2}$;
	\end{enumerate}
then $N^m$ is homotopy equivalent to $S^m$, $\RP^m$,  $\CP^{\frac m 2}$, or a lens space.
In particular, $M^{m+2}$ is homotopy equivalent to $S^{m+2}$, $\RP^{m+2}$,  $\CP^{\frac{m+2}{2}}$, or a lens space.
\end{Codim2}

We now recall the Borel formula from Borel \cite{B}. 

\begin{Borel} \label{borel} \cite{B}
Let $p>2$ be a prime and let $\mathbb{Z}_p^r$ act smoothly on $M$, a Poincaré duality space, with fixed-point set component $F$. Then
$$
\codim (F \subseteq M) \;=\; \sum \codim(F \subseteq F'),
$$
where the sum runs over the fixed-point set components $F'$ of corank one subgroups 
$\mathbb{Z}_p^{r-1} \subseteq \mathbb{Z}_p^{r}$ for which $\dim(F')>\dim(F)$.

These subgroups are precisely the kernels of the irreducible subrepresentations of
the isotropy representation of $\mathbb{Z}_p^r$ on the normal space to $F$. 
In particular, the number of pairwise inequivalent irreducible subrepresentations 
is at least $r$ whenever the action is effective. Moreover, equality holds only if 
the isotropy representation is equivalent to a block diagonal representation of the form
\[
(\eta_1,\dots,\eta_r) \;\longmapsto\;
\operatorname{diag}\bigl(R(\eta_1)^{m_1},\dots,R(\eta_r)^{m_r}\bigr),
\]
where $m_i>0$ denotes the multiplicity of the $i$-th irreducible summand, and
$R(\eta_i)$ is the $2\times 2$ real rotation matrix corresponding to
multiplication by a primitive $p$-th root of unity $\eta_i$. 
Here $R(\eta_i)^{m_i}$ denotes the block diagonal matrix consisting of $m_i$ copies of $R(\eta_i)$.
\end{Borel}

We also record a similar useful observation to Observation 3.6 in \cite{KKSS}. The same proof as in \cite{KKSS}, a consequence of the \hyperref[CL]{Connectedness Lemma} and the \hyperref[borel]{Borel formula}, works here, as the proof does not depend on the value of the prime $p$. 
\begin{observe}\cite{KKSS}\label{borelobserve}
    Let $F_j^{m_j}$ be the fixed-point set of a $\Z_p^{r-j}$-action by isometries on a closed, positively curved manifold, and suppose that $r-j \geq 2.$ If the isotropy representation has exactly $r-j$ irreducible subrepresentations, and if the fixed-point set component $F_{j+1}$ containing $F_j$ of some $\Z_p^{r-j-1}$ has the property that the codimension $k_j$ of the inclusion $F_j \subseteq F_{j+1}$ is minimal, then this inclusion is $\dim(F_j)$-connected.  
\end{observe}

 \subsection{Bounds for \texorpdfstring{${\bf \Z_p}$}--tori in Positive Curvature}

Fang and Rong 
showed in \cite{FR} that for a sufficiently large prime $p$,
the maximal $\Z_p$-symmetry rank of a closed, simply-connected $2n$-dimensional manifold of positive curvature  is $n$.   
Ghazawneh in \cite{Gh} showed that if instead we assume that the $\Z_p$-torus admits a fixed point, with $p>2$ and an additional technical assumption on fixed point sets of dimension $4$, then for a manifold of {\em any} dimension $n$ the maximal symmetry rank is $\lfloor n/2\rfloor$.

 In \cite{FR}, using their maximal symmetry rank result, they also establish a homeomorphism classification result for such actions on
 a closed, simply-connected even-dimensional manifold of positive curvature.

The following statement is weaker than the corresponding result of \cite{Gh}, but we include it here in order to anchor the induction used in Theorem \ref{Gh+Zp}, where we remove the technical assumption in the result of \cite{Gh}.
\begin{theorem} \label{n/2result}
    Let $M^n$, $n\geq5$ be a closed Riemannian manifold of positive sectional curvature. Suppose that $M$ admits an isometric and effective $\Z_p^k$-action with a fixed point $x \in M, p>2$. Suppose that $k\geq \lfloor n/2\rfloor$, then $k=\floor{n/2}$, and  $M$ is homotopy equivalent to $S^n$, $\RP^n$, $\CP^{n/2}$, or a lens space.
\end{theorem}

\begin{proof}
    We first note that for odd dimensions, the result follows directly from Part 1 of the \hyperref[codim2]{Codimension Two Lemma}. 
For even dimensions, we note that  we have the following chain of nested fixed point sets
$$F^{n-4}\subset F^{n-2}\subset M,$$
where $F^{n-4}$ is fixed by $\Z_p^2$ and $F^{n-2}$ is fixed by some $\Z_p\subset \Z_p^2$. Using the \hyperref[borel]{Borel Formula} and Observation \ref{borelobserve}, it follows that $F^{n-4}$ is $(n-2)$-connected in $F^{n-2}$, and by Part 2 of the \hyperref[codim2]{Codimension Two Lemma}, we see that $F^{n-2}$ is homotopy equivalent to $S^{n-2}$, $\RP^{n-2}$, or $\CP^{(n-2)/2}$.
Note that Part~1 of the \hyperref[CL]{Connectedness Lemma} implies that the inclusion $F\hookrightarrow M$ is $(n-3)$-connected; in particular, it is $2$-connected for $n\ge 5$. Because $2\le \frac{n+3}{4}$ when $n\ge 5$ and $N$ has the cohomology of $S^{n-2}$, $\RP^{n-2}$, or $\CP^{(n-2)/2}$, we may apply Corollary~\ref{corwilk} with $k=2$ to conclude that $M$ has the cohomology ring of $S^{n}$, $\RP^{n}$, or $\CP^{n/2}$. Finally, lifting to the universal cover $\widetilde{M}$ of $M$ and using Theorem~\ref{ucover}, we obtain the desired conclusion.

\end{proof}

As mentioned earlier, the following theorem as stated was proven in \cite{Gh} under an extra assumption on the Euler characteristic of any $4$-dimensional fixed point sets of a $\Z_p$-subtorus contained in a chain of fixed point sets each of codimension two in the next, which was required to prove the anchor of the induction. Here we are able to remove that condition. The remainder of the proof is similar to that of \cite{Gh1}, but we include it for the sake of completeness. 
The proof techniques here make use of the tools developed in \cite{KKSS}, and can also be used to significantly streamline the proof of the even-dimensional case in \cite{FR}.

\begin{theorem} \label{Gh+Zp}
Let $M^n$, $n\geq 13$ be a closed Riemannian manifold of positive sectional curvature. Suppose that $M$ admits an isometric and effective $\Z_p^r$-action with a fixed point $x\in M$, $p>2$. If
$$
r\,\geq \begin{cases}
   \lfloor \frac{3n}{8}\rfloor +2 \,\, \textrm{if }\,\, n\equiv 2, 4 \mod{8}\\
   \lfloor \frac{3n}{8}\rfloor +1 \,\, \textrm{, }  
\end{cases}
$$
then $M$ is homotopy equivalent to $S^n$, $\RP^n$, $\CP^{n/2}$, or a lens space.
\end{theorem}
\begin{proof}

We prove the result for $n\geq 13$, the first value of $n$ for which the bound in the statement of the theorem is different than the maximal symmetry rank, by total induction on the dimension of $M$. We first establish the anchor of the induction for dimensions $13\leq n\leq 16$.
Let $\Z_p^r$ act on $M^n$. Let $N$ be the maximal component of the fixed point set of a $\Z_p$ subgroup of $\Z_p^r$ that contains the fixed point of the $\Z_p^r$-action. Then, by maximality, $N$ admits an effective, isometric $\Z_p^{r-1}$-action with a fixed point. Since the maximal symmetry rank for any $\Z_p^r$ action on an $n$-manifold is $\lfloor n/2\rfloor$, it follows that $\dim(N)\geq 2(r-1)$. In particular, for dimensions $14$ and $16$, this means that 
$$n-4\leq \dim(N)\leq n-2.$$
In each of these dimensions, there are then two cases to consider: Case 1, where $\dim(N)=n-2$ and Case 2, where $\dim(N)=n-4$.

\noindent{\bf Case 1}: where $\dim(N)=n-2$. When $n$ is odd, the result follows immediately by Part 1 of the 
\hyperref[codim2]{Codimension Two Lemma}. We consider dimensions $14$ and $16$, simultaneously. Let $F$ be a maximal component of the fixed-point set of $\Z_p^2$ that contains the fixed point and is contained in $N$. Since $N$ is of codimension two in $M$, it follows from the \hyperref[borel]{Borel Formula}, that there is another fixed point set component $N'$ corresponding to a distinct $\Z_p$ subgroup of $\Z_p^2$ such that $\codim_{N'}(F)=2$. By Observation \ref{borelobserve}, it follows that $F\hookrightarrow N'$ is $\dim(F)$-connected and hence $F$ and $N'$ are both spheres, real projective spaces, complex projective spaces, or lens spaces. Since $N'\hookrightarrow M$ is at least $3$-connected,  applying Lemma \ref{corwilk} gives us the desired result.

\noindent{\bf Case 2}: where $\dim(N)= n-4$. But in both cases, the induced action of $\Z_p^{r-1}$ on $N$ is of rank $\dim(N)/2$, and hence by Part 1, $N$ must be homotopy equivalent to $S^{n-4}$, $\RP^{(n-4)/2}$, or $\CP^{(n-4)/2}$. Since $4\leq (n-1)/4$ in both cases and hence $N\hookrightarrow M$ is at least $3$-connected in both cases by Part 1 of the \hyperref[CL]{Connectedness Lemma}, we may use Lemma \ref{corwilk} to see that $M$ is as desired.

This proves the induction anchor. Now, let $N$ be defined as above.
Note that $2(r-1)\leq \dim(N)\leq n-2$. For this range of dimensions, $\codim(N)\leq (n+3)/4$. Then $N$ with the induced $\Z_p^{r-1}$-action satisfies the induction hypothesis for the following two cases: case (a), where $n\not\equiv 6 \pmod 8$, and case (b), where $n\equiv 6 \pmod 8$ and $\dim(N)\leq n-4$. Hence in all these cases, $N$ is homotopy equivalent to a sphere, a real projective space, a complex projective space, or a lens space and the desired result for case (a) follows immediately by Lemma \ref{corwilk}, noting that $N\hookrightarrow M$ is at least $(n-5)/2\geq 2$-connected.

It remains to consider the case where $n\equiv 6 \pmod 8$ and $\dim(N)=n-2$, but the proof of this case proceeds exactly as in Case~1 of the induction anchor.

\end{proof}
\begin{remark}
    Theorem \ref{Gh+Zp} above holds for $n\geq 5$, but we state it for $n\geq 13$, since 
for $5 \leq n \leq 12$, the bound in the statement of the theorem equals $\lfloor \frac{n}{2} \rfloor$, which is the maximal symmetry rank. Hence in this range we have $r=\lfloor \frac{n}{2} \rfloor$, and the desired result follows from Theorem~\ref{n/2result}.
\end{remark}

\section{The Proof of Theorem B}\label{4}

We begin by stating an important lemma that gives us two distinct elements of $\Z_p^r$ fixing closed submanifolds of sufficiently small codimension.

Throughout this section, we define the following constants, depending on $p$, for primes $3\leq p\leq 29$.

\begin{align}\label{constants_f}
f_i(p)=\begin{cases}
   \dfrac{1}{2} (1- H_p(J_p(\tfrac14))) & i=1, \\[6pt]
   \log_p\!\left( 
      \dfrac{p(p - 1)(1 - J_p(\tfrac14))(2\pi J_p(\tfrac14))^{1/2}}
            {4J_p(\tfrac14)} 
   \right)
   + \dfrac{2}{13 \ln p} - 2.5 \log_p 2 & i = 2, \\[10pt]
   \dfrac{1}{6\ln p} + \dfrac{1}{\ln p\,(1 - J_p(\tfrac14))} & i=3,\\[6pt]
   \dfrac{1}{\ln p} & i=4,\\[6pt]
   J_p(\tfrac14) & i=5,
\end{cases}
\end{align}

\begin{lemma} \label{lemmaA}
Assume $\Z_p^r$ acts effectively and isometrically on a closed, positively
curved manifold $M^n$ with a non-empty fixed-point set, where $p \in \{3,5,7,11,13,17,19,23,29\}$ and $n \geq 16$. 
Suppose
\[
r \;>\; f_1(p)\,n \;+\; 2.5 \log_p n \;+\; f_2(p)
   \;+\; \frac{f_3(p)}{\,n-1\,}
   \;+\; \frac{f_4(p)}{\,f_5(p)(n-1) - 2},
\]

Then the following hold:
\begin{enumerate}
\item there exists $\tau_1\in\Z^r_p$ such that $\codim(M^{\tau_1})\leq \frac{n+3}{4}$ ; and
\item there exists $\tau_2\in  \Z^r_p$ such that $\codim(M^{\tau_2})
\leq \frac{n+1}{3}$ , where $\tau_2\not\in\langle\tau_1\rangle$.
\end{enumerate}
\end{lemma}
\begin{proof}
The proof follows along the lines of Lemma 3.1 in \cite{Gh1}, which we include for the convenience of the reader. 

Let $x \in M$ be a fixed point of the $\Z_p^r$-action. The isotropy representation induces an injective homomorphism
$$
\rho : \Z_p^r \to \Z_p^m,
$$
describing the action of $\Z_p^r$ on the tangent space $T_x M$. For each nontrivial $\tau \in \Z_p^r$, the fixed-point set $M^{\tau}$ is a totally geodesic submanifold of $M$, and its codimension at $x$ is given by:
\vspace{-0.2cm}
$$
\operatorname{codim} M^{\tau}_x = 2|\rho(\tau)|,
$$
where $|\rho(\tau)|$ is the Hamming weight of the vector $\rho(\tau) \in \Z_p^m$ and the factor of 2 comes from the fact that the representations of $\Z_p$ are even-dimensional.

We argue by contradiction. Suppose then that
$$
\codim(M^{\tau}) > (n+3)/4, \qquad \forall \tau \in \Z_p^r \setminus \{0\}.
$$
Setting $\delta = \frac{1}{4}$, this corresponds to vectors with weight at least $\delta m$, where $m=\floor{n/2}$. Applying Corollary \ref{rEB(p)} with $\delta = 1/4$, we obtain that
\begin{align*}
r &\leq \left(1 - H_p\left(J_p(\frac{1}{4})\right)\right) \floor{\frac{n}{2}}
+ 2.5\log_p \left(\floor{\frac{n}{2}}\right) 
+ \log_p\!\left( \frac{(p - 1)(1 - J_p\left(\frac{1}{4}\right))}{J_p\left(\frac{1}{4}\right)} \right) \\
&\,\,\,\,\,\,\,\,+ 0.5 \log_p\!\left( 2\pi J_p\left(\frac{1}{4}\right) \right) 
+ \log_p\left(\frac{p}{4}\right) 
+\frac{2}{13 \ln p} \\
&\,\,\,\,\,\,\,\,+ \frac{1}{\floor{\frac{n}{2}}}\left(\frac{1}{12\ln p}
   +\frac{1}{(2 \ln p)(1 - J_p\left(\frac{1}{4}\right))}\right) 
+ \frac{1}{(2\ln p)J_p\left(\frac{1}{4}\right)\floor{\frac{n}{2}} - 2 \ln p}.
\end{align*}
a contradiction to the assumption on $r$.
Thus, there exists $\tau_1 \in \Z_p^r$ such that:
\[
\codim M^{\tau_1}_x \leq \frac{n+3}{4}.
\]This completes the proof of Part (1).

The proof of Part (2) is similar, except that we set $\delta=\frac{1}{3}$ and use 
Proposition \ref{prop:bound-decreasing} to get our result.

\end{proof}
\vspace{-.75cm}
\begin{remark} \label{remA}
    Assuming the inductive hypothesis in the proof of the subsequent theorem, in the situation of Lemma \ref{lemmaB}, if $\pi_1(M)=0$, $N_1^{n-k_1} \subset M_x^{\tau_1}$, and $N_2^{n-k_2} \subset M_x^{\tau_2}$, then $k_1$ and $k_2$ satisfy the hypothesis of Parts (1) and (2) of Proposition \ref{cohm2}.
\end{remark}

We now have all the tools needed to prove Theorem \ref{main}. We restate the theorem for the convenience of the reader.

\begin{theorem}
Assume $\Z_p^r$ acts effectively and isometrically on a closed, positively
curved manifold $M^n$ with a fixed-point $x\in M$, where $p \in \{3,5,7,11,13,17,19,23,29\}$ and $n \geq 16$. 
Suppose
$$
r \;>\; f_1(p)\,n \;+\; 2.5 \log_p n \;+\; f_2(p)
   \;+\; \frac{f_3(p)}{\,n-1\,}
   \;+\; \frac{f_4(p)}{\,f_5(p)(n-1) - 2}.
$$
Then $M$ is homotopy equivalent to $S^n$, $\RP^n$, $\CP^{n/2}$, or a lens space.
\end{theorem}

\begin{proof}
    This is a streamlined version of the proof used to prove the Main Theorem in \cite{Gh1}, with a few minor modifications.  We include it here for the convenience of the reader. The proof is by total induction on dimension. The anchor of the induction is for dimensions $\leq N(p)$, where $N(p)$ is given in the table below.
    \begin{table}[h] \label{tbl2}
\centering
\renewcommand{\arraystretch}{2.0}
\begin{adjustbox}{max width=\textwidth}
\begin{tabular}{|c||c|c|c|c|c|c|c|c|c|c|c|c|c|c|}
\hline
$p$
& 3 & 5 & 7 & 11 & 13 & 17 & 19 \\ 
\hline\hline
$N(p)$
& $91$
& $123$
& $163$
& $259$
& $351$
& $531$
& $699$\\
\hline
\end{tabular}
\end{adjustbox}
\caption{Values for $N(p)$ for small primes}
\label{tab:placeholder2}
\end{table}

For $n\leq N(p)$, it can be checked that $$\; f_1(p)\,n \;+\; 2.5 \log_p n \;+\; f_2(p)
   \;+\; \frac{f_3(p)}{\,n-1\,}
   \;+\; \frac{f_4(p)}{\,f_5(p)(n-1) - 2}$$ $$> \begin{cases}
   \lfloor \frac{3n}{8}\rfloor +2 \,\, \textrm{if }\,\, n\equiv 2, 4 \mod{8}\\
   \lfloor \frac{3n}{8}\rfloor +1 \,\, \textrm{otherwise, }\\
\end{cases}$$ and thus our result holds by Theorem \ref{Gh+Zp}. So we assume $n>N(p)$ for $3\leq p\leq 19$. We suppose our result holds for all dimensions less than $n$ and show that it also holds for dimension $n$. Let $x$ be a fixed-point of the $\mathbb{Z}_p^r$-action.

By Part (1) of Lemma \ref{lemmaA}, there exists a non-trivial $\tau_1 \in \mathbb{Z}_p^r$ and $x \in N_1^{n-k_1} \subset$ Fix($M;\tau_1$) such that $k_1 \leq \frac{n+3}{4}$ and $N_1$ is of maximal dimension. By maximality of $N_1$, the kernel of the $\mathbb{Z}_p^{r}$-action is $\mathbb{Z}_p$ and hence $N_1$ admits an effective $\mathbb{Z}_p^{r-1}$-action with $x \in N_1$.

By Part (2) of Lemma \ref{lemmaA}, there exists a second non-trivial $\tau_2 \in \mathbb{Z}_p^r$ such that $\tau_2 \neq \langle\tau_1\rangle$ and $x \in N_2^{n-k_2} \subset$ Fix($M;\tau_2$) such that $k_2 \leq \frac{n+1}{3}$ and $N_2$ is of maximal dimension. By maximality of $N_2$, the kernel of the $\mathbb{Z}_p^r$-action is $\mathbb{Z}_p$ and hence $N_2$ admits an effective $\mathbb{Z}_p^{r-1}$-action with $x \in N_2$. Without loss of generality, assume that $k_1 \leq k_2$. We now proceed to argue according to the codimension of $N_2$. We have two cases (since the fixed-point sets of $\Z_p^r$ have even codimension): Case 1, where $k_2 \geq 4$; Case 2, where $k_2 = 2$. We begin with Case 1.
    
\noindent\textbf{Case 1: $k_2 \geq 4$}. Recall that $\mathbb{Z}_p^{r-1}$ acts effectively and isometrically on $N_2^{n-k_2}$.

Using the fact that $$r > \; f_1(p)\,n \;+\; 2.5 \log_p n \;+\; f_2(p)
   \;+\; \frac{f_3(p)}{\,n-1\,}
   \;+\; \frac{f_4(p)}{\,f_5(p)(n-1) - 2}$$ and $k_2\geq 4$, a computation gives
   $$r-1 > \; f_1(p)\,(n-k_2) \;+\; 2.5 \log_p (n-k_2) \;+\; f_2(p)
   \;+\; \frac{f_3(p)}{\,(n-k_2)-1\,}
   \;+\; \frac{f_4(p)}{\,f_5(p)((n-k_2)-1) - 2} $$

Hence, the induction hypothesis is satisfied and it then follows that $N_2$ is homotopy equivalent to $S^{n-k_2}$, $\RP^{n-k_2}$, $\CP^{(n-k_2)/2}$, or a lens space. Then the universal cover of $N_2$ satisfies the hypotheses of Lemma \ref{combo}and the result follows. This completes the proof of  Case 1.

\noindent\textbf{Case 2: $k_2 \leq 2$.} Since $k_1 \leq k_2$, it follows that $k_1=k_2=2$. As we saw in the proof of Case 1 of Theorem \ref{Gh+Zp}, it then follows that $N_1$ and $N_2$ intersect transversely.

From Observation \ref{borelobserve}, we have that $N_1 \cap N_2$ is dim($N_1 \cap N_2$)-connected in $N_2$. It then follows from Part (2) of the Codimension Two Lemma \ref{codim2} that $M$ is homotopy equivalent to $S^n$, $\RP^n$, $\CP^{n/2},$ or a lens space.
This completes the proof of Case 2. 
\end{proof}

\section{The Proof of Theorem A}\label{5}

As in the previous section, we begin by proving an important lemma that gives us two distinct elements of $\Z_p^r$ fixing closed submanifolds of sufficiently small codimension.

\begin{lemma} \label{lemmaB}
Assume $\Z_p^r$ acts effectively and isometrically on a closed, positively
curved manifold $M^n$ with a non-empty fixed-point set, where $p \ge 3$ and $n \geq 13$. 
Suppose
\[
r \;>\; \frac{3p-4}{8(p-1)}\,n \;+\; \log_p(n) \;+\; \log_p\!\left(\frac{p^{3}}{8}\right).
\]

Then the following hold:
\begin{enumerate}
\item there exists $\tau_1\in\Z^r_p$ such that $\codim(M^{\tau_1})\leq \frac{n+3}{4}$ ; and
\item there exists $\tau_2\in  \Z^r_p$ such that $\codim(M^{\tau_2})
\leq \frac{n+1}{3}$ , where $\tau_2\not\in\langle\tau_1\rangle$.
\end{enumerate}
\end{lemma}

\begin{proof}
The argument is identical to that of Lemma~\ref{lemmaA}, with the only
modification being the substitution of the Plotkin estimate in Corollary \ref{rPlotkin(p)} at $\delta=\tfrac14$ and $\delta=\tfrac13$.  Under the
numerical hypothesis
\[
r \;>\; \frac{3p-4}{8(p-1)}\,n \;+\; \log_p(n) \;+\;
\log_p\!\left(\frac{p^{3}}{8}\right),
\]
by Proposition \ref{monotone2} and
 as in the proof of Lemma~\ref{lemmaA}, the bounds $r_{\mathrm{Pl}}(1/4)$ and $r_{\mathrm{Pl}}(1/3)$ guarantee the existence of an element
$\tau_1\in\Z_p^r$ with $\codim(M^{\tau_1})\le (n+3)/4$, and a second
element $\tau_2\notin\langle\tau_1\rangle$ with
$\codim(M^{\tau_2})\le (n+1)/3$.  The remainder of the proof is exactly the same as the proof of Lemma \ref{lemmaA}.
\end{proof}

We now have all the tools needed to prove Theorem \ref{main2}. We restate the theorem for the convenience of the reader.

\begin{theorem}
Assume $\Z_p^r$ acts effectively and isometrically on a closed, positively
curved manifold $M^n$ with a fixed-point $x\in M$, where $p \ge 3$ and $n \geq 13$. 
Suppose
\begin{equation}\label{eq:plotkin-rank}
r \;>\; \frac{3p-4}{8(p-1)}\,n \;+\; \log_p(n) \;+\; \log_p\!\left(\frac{p^{3}}{8}\right).
\end{equation}
Then $M$ is homotopy equivalent to $S^n$, $\RP^n$, $\CP^{n/2}$, or a lens space.
\end{theorem}

\begin{proof}
By Proposition \ref{LambertW}, for each prime $p \ge 3$
and every integer $13 \le n \le \lfloor n_{+}(p)\rfloor$, the bound in
Theorem~\ref{Gh+Zp} is smaller than the Plotkin-based rank
assumption~\eqref{eq:plotkin-rank}.
Thus, throughout this finite range of dimensions, the hypothesis
\eqref{eq:plotkin-rank} automatically implies the hypothesis of
Theorem~\ref{Gh+Zp}, providing the base of the induction.

For $n > \lfloor n_{+}(p)\rfloor$, the remainder of the argument proceeds
exactly as in the proof of Theorem~\ref{main}, with
Lemma~\ref{lemmaB} used in place of Lemma~\ref{lemmaA}.  This completes
the proof.
\end{proof}


\begin{thebibliography}{999}
\bibitem{Ba} Y.V. Baza\u{\i}kin, {\em A Manifold with Positive Sectional Curvature and Fundamental Group $\Z_3\oplus \Z_3$}, Sib. Math. J., {\bf 40} (1999), 834--836.
\bibitem{Be} M. Berger, {\em Trois remarques sur les vari\'et\'es riemanniennes \`a courbure positive}, C. R. Math. Acad. Sci. Paris, {\bf 263} (1966), A76--A78.
\bibitem{B} A. Borel, {\em Fixed Point Theorems for Elementary Commutative Groups II},  Seminar on Transformation Groups, {\bf 46}, Annals of Mathematics Series, Princeton University Press (1961).
\bibitem{CorlessLambertW}
R.~M.~Corless, G.~H.~Gonnet, D.~E.~G.~Hare, D.~J.~Jeffrey, and D.~E.~Knuth,
\emph{On the Lambert $W$ function},
Advances in Computational Mathematics \textbf{5} (1996), 329--359.
\bibitem{F} F. Fang, {\em Finite isometry groups of $4$-manifolds with positive sectional curvature}, 
Math Z., {\bf 259} (2008), 643--656, DOI 10.1007/s00209-007-0242-0.
\bibitem{FR2} F. Fang and X. Rong, {\em Homeomorphism classification of positively curved manifolds with almost maximal symmetry rank}, Math. Ann. 332, 81-101 (2005). https://doi.org/10.1007/s00208-004-0618-y.

\bibitem{FR} F. Fang and X. Rong, {\em Positively curved manifolds with maximal discrete symmetry
rank}, American Journal of Mathematics, {\bf 126} no. 2, 227--245 (2004).
\bibitem{FRW} P. Frank, X. Rong, and Y. Wang, {\em Fundamental groups of positively curved manifolds with symmetry}, Math. Ann., {\bf 355}  (2013), 1425--1441.
\bibitem{Gh1} F. Ghazawneh, {\em $\Z_2$-actions on positively curved manifolds}, preprint:arxiv:2409.15392 (2024).
\bibitem{Gh} F. Ghazawneh, {\em Positive Curvature and Discrete Abelian Group Actions}, PhD Thesis (2025).

\bibitem{GS} K. Grove and C. Searle,  {\em Positively curved manifolds with maximal symmetry rank}, Jour. of Pure and Appl. Alg. {\bf 91}  (1994), 137--142.
\bibitem{GSh} K. Grove and K. Shankar, {\em Rank two fundamental groups of positively curved manifolds},  J. Geom Anal {\bf 10}, 679--682 (2000). https://doi.org/10.1007/BF02921991


\bibitem{Hi} A. Hicks    {\em Group Actions and the Topology of Nonnegatively Curved 4-Manifolds},  Illinois J. Math. {\bf 41}  no. 3 (1997) 421--437.
\bibitem{Guruswami}
V. Guruswami,
\emph{List Decoding of Error-Correcting Codes},
Available at \url{https://www.cs.cmu.edu/~venkatg/pubs/papers/frozen.pdf} (2001).


\bibitem{KKSS} L. Kennard, E. Khalili Samani, and C. Searle, {\em Positively curved manifolds with discrete abelian symmetries}, preprint:arxiv:2110.13345v2 (2024). 

\bibitem{KWW1} L. Kennard, M. Wiemeler, and B. Wilking, {\em Splitting of torus representations and applications in the Grove symmetry program}, preprint arxiv:2106.14723  (2021).
\bibitem{KWW2} L. Kennard, M. Wiemeler, and B. Wilking, {\em Positive curvature, torus symmetry, and matroids}, preprint arxiv:2212.08152 (2022)

\bibitem{KL}  J. H. Kim, H. K. Lee, {\em On non-negatively curved 4-manifolds with discrete symmetry},  Acta Math. Hungar., {\bf 125} (3) (2009), 201--208. DOI: 10.1007/s10474-009-8237-4. 

\bibitem{MRRW}
R.~J.~McEliece, E.~R.~Rodemich, H.~C.~Rumsey, Jr., and L.~R.~Welch,
``New Upper Bounds on the Rate of a Code via the Delsarte--MacWilliams Inequalities,''
\emph{IEEE Trans.\ Inform.\ Theory} 23 (1977), no.~2, 157--166.

\bibitem{Robbins} H. Robbins, {\em A Remark on Stirling's Formula}, The American Mathematical Monthly, 62 (1): 26-29, doi:10.2307/2308012  (1955).
\bibitem{RS05} X. Rong and X. Su, {\em The Hopf Conjecture for Manifolds with Abelian Group Actions}, 
Comm. Contemp. Math. {\bf 7} no. 1 (2005)  121--136.

\bibitem{RudraLect16}
A.~Rudra,
\newblock {\em Lecture 16: Plotkin Bound,}
\newblock Error Correcting Codes: Combinatorics, Algorithms and Applications,  
University at Buffalo (SUNY), Fall 2007.
\newblock Available at: \url{https://cse.buffalo.edu/faculty/atri/courses/coding-theory/lectures/lect16.pdf}.

\bibitem{RudraLect14}
A.~Rudra,
\emph{Lecture 14: List Decoding Capacity},
Error Correcting Codes: Combinatorics, Algorithms and Applications (Fall 2007), 
University at Buffalo, 2007.
Available at \url{https://cse.buffalo.edu/faculty/atri/courses/coding-theory/lectures/lect14.pdf}.
\bibitem{Ru3}
A.~Rudra,
\emph{Lecture 17: Proof of a Geometric Lemma},
Error Correcting Codes: Combinatorics, Algorithms and Applications (Fall 2007), 
University at Buffalo, 2007.
Available at \url{https://cse.buffalo.edu/faculty/atri/courses/coding-theory/lectures/lect17.pdf}.
\bibitem{RudraLect19}
A.~Rudra,
\emph{Lecture 19: Elias-Bassalygo Bound},
Error Correcting Codes: Combinatorics, Algorithms and Applications (Fall 2007), 
University at Buffalo, 2007.
Available at \url{https://cse.buffalo.edu/faculty/atri/courses/coding-theory/lectures/lect19.pdf}.

\bibitem{SW} X. Su and Y. Wang, {\em The Hopf Conjecture for positively curved manifolds with discrete abelian group actions},
Differential Geometry and its Applications, vol. 26, issue 3  (2008), pp. 313--322.
\bibitem{Su} K. Sugahara, {\em The isometry group and the diameter of a Riemannian manifold with positive
curvature}, Math. Japon., {\bf 27} (1982), 631--634.
\bibitem{W} Y. Wang, {\em Fundamental groups of closed positively curved manifolds with
              almost discrete abelian group actions},
 Acta Mathematica Scientia. Series B. English Edition, {\bf 30}, no. 1 (2010),
   203--207 {\scriptsize\url{https://doi.org/10.1016/S0252-9602(10)60037-9}}.


\bibitem{Wilkingsurvey} B. Wilking {\em Nonnegatively and Positively Curved Manifolds}, preprint arXiv:0707.3091 (2007).
\bibitem{Wilk} 
B. Wilking  {\em Torus actions on manifolds of positive sectional curvature},  
Acta Math., {\bf 191}  no. 2  (2003) 259--297. 

\bibitem{Y} D. G. Yang  {\em On the Topology of Nonnegatively Curved Simply Connected 4-manifolds 
with Discrete Symmetry}, 
Duke Math. J., {\bf 74}  (1994)  531--545.




\end{thebibliography}
\end{document}